\documentclass{amsart}
\usepackage{amssymb, amsmath, latexsym, mathabx, dsfont}

\renewcommand{\baselinestretch}{\baselinestretch}
\renewcommand{\baselinestretch}{1.1}
\numberwithin{equation}{section}

\newtheorem{thm}{Theorem}[section]
\newtheorem{lem}[thm]{Lemma}
\newtheorem{cor}[thm]{Corollary}

\theoremstyle{definition}

\theoremstyle{remark}
\newtheorem{rmk}[thm]{Remark}

\newcommand{\gen}{\text{gen}}
\newcommand{\spn}{\text{spn}}

\newcommand{\z}{{\mathbb Z}}

\newcommand{\mbw}{{\mathcal M_{B,\mathbf w}^{\mathbf s}}}
\newcommand{\mbz}{{\mathcal M_{B,\mathbf 0}^{\mathbf s}}}

\numberwithin{equation}{section}

\begin{document}
\title{Spinor representations of positive definite ternary quadratic forms}

\author{Jangwon Ju, Kyoungmin Kim and Byeong-Kweon Oh }

\address{Department of Mathematical Sciences, Seoul National University, Seoul 08826, Korea}
\email{jjw@snu.ac.kr}

\thanks{This work of the first author was supported by BK21 PLUS SNU Mathematical Sciences Division.}

\address{Department of Mathematical Sciences, Seoul National University, Seoul 08826, Korea}
\email{kiny30@snu.ac.kr}

\address{Department of Mathematical Sciences and Research Institute of Mathematics, Seoul National University, Seoul 08826, Korea}
\email{bkoh@snu.ac.kr}
\thanks{This work of the third author was supported by the National Research Foundation of Korea (NRF-2014R1A1A2056296).}



\begin{abstract} For a positive definite integral ternary quadratic form $f$, let $r(k,f)$ be the number of representations of an integer $k$ by $f$.  The famous Minkowski-Siegel formula implies that if the class number of $f$ is one, then $r(k,f)$ can be written as a constant multiple of a product of local densities which are easily computable.  In this article, we consider the case when the spinor genus of $f$ contains only one class. In this case the above also holds if $k$ is not contained in a set of finite number of square classes which are easily computable (see, for example,  \cite{sp1} and \cite {sp2}).  By using this fact, we prove some extension of the results given in both \cite {cl} on the representations of generalized Bell ternary forms and   \cite {be} on the representations of  ternary quadratic forms with some congruence conditions.    
\end{abstract}

\maketitle

\section{introduction}
For a positive definite integral ternary quadratic form 
$$
f(x,y,z)=ax^2+by^2+cz^2+2pyz+2qzx+2rxy \  \ (a,b,c,p,q,r \in \z),
$$ 
and an integer $k$, we define $r(k,f)$  the number of representations of $k$ by $f$, that is, 
$$
r(k,f)=\vert \{ (x,y,z) \in \z^3 : f(x,y,z)=k\}\vert.
$$
Note that it is always finite, for we are assuming that $f$ is positive definite.  It seems to be quite a difficult problem to compute $r(k,f)$ effectively
for an arbitrary ternary quadratic form $f$.  The famous Minkowski-Siegel formula tells us that the weighted sum of the number of representations by the genus of $f$ is a constant multiple of  the product of local densities. To be more precise, we define 
$$
w(\gen(f))=\sum_{[f'] \in \gen(f)} \frac 1 {o(f')} \quad \text{and} \quad r(k,\gen(f))=\frac 1 {w(\gen(f))}\sum_{[f'] \in \gen(f)} \frac {r(k,f')}{o(f')}.
$$
Here $o(f')$ is the order of the isometry group of $f'$ and $[f']$ is the set of isometric classes containing the quadratic form $f'$. The Minkowski-Siegel formula says that
\begin{equation} \label{minsi}
r(k,\gen(f))=2\pi \sqrt{\frac {k}{df}} \times \prod_{p}\alpha_p({k,f}),
\end{equation}
where $\alpha_p$ is the local density depending only on the  structure of $f$ over $\z_p$ and an integer $k$. In particular, if the class number of $f$ is one, then 
$r(k,f)=r(k,\gen(f))$. 

 As a natural modification of the    Minkowski-Siegel formula, it was proved in \cite {k} and \cite {w} that if we define
$$
w(\spn(f))=\sum_{[f'] \in \spn(f)} \frac 1 {o(f')} \quad \text{and} \quad r(k,\spn(f))=\frac 1 {w(\spn(f))}\sum_{[f'] \in \spn(f)} \frac {r(k,f')}{o(f')},
$$
then $r(k,\spn(f))=r(k,\gen(f))$ for any integer $k$ that is not a splitting integer. In \cite {sp2}, Schulze-Pillot  gave  a formula for $r(k,\spn(f))-r(k,\spn(f'))$ for any $f'$ in the genus of $f$ and any splitting integer $k$.  Hence if there is only one class in the spinor genus of $f$, then we may have some formula on $r(k,f)$  by using $r(k,\gen(f))$. 

Recently, some interesting formulas on the number of representations of ternary  quadratic forms were proved by using some identities of $q$-series. The aim of this article is to reprove and extend some of those results by using spinor representation theory.    

In Section 2, we consider the representations of a generalized Bell ternary form defined by $f_{\alpha,\beta}(x,y,z)=x^2+2^{\alpha}y^2+2^{\beta}z^2 \ (\alpha \le \beta)$, where $\alpha$ and $\beta$ are  non negative integers. In \cite{b}, Bell gave a formula on the representations of integers by the form $f_{\alpha,\beta}(x,y,z)$, where $(\alpha,\beta) \in \{1,2,4,8\}$.  Recently,  H{\"u}rlimann  in \cite{cl} provided a formula, by using $q$-series identities, on the numbers of representations of $f_{\alpha,\beta}$, where $(\alpha,\beta)=(2,16), (8,16)$, and resolved Cooper and Lam's conjecture given in \cite{cl2} in these two cases.  In fact, the class numbers of those forms considered by Bell and these two forms are all one. Hence one may have  the same results by using the Minkowski-Siegel formula as follows:  for an integer $a$, we define $\text{sgn}(a)=1$ if $a$ is odd,  $\text{sgn}(a)=0$ otherwise, and $\widetilde{f_{\alpha,\beta}}(x,y,z)=x^2+y^2+2^{\text{sgn}(\alpha+\beta)}z^2$. Then one may easily show that $f_{\alpha,\beta}$ is isometric to $\widetilde{f_{\alpha,\beta}}$ over $\z_p$ for any odd prime $p$. Hence the Minkowski-Siegel formula  (\ref{minsi}) implies that 
\begin{equation} \label{minsi-3}
r(k,\gen(f_{\alpha,\beta}))=\frac 1{\sqrt{2^{\alpha+\beta-\text{sgn}(\alpha+\beta)}}}\cdot \frac {\alpha_2(k,f_{\alpha,\beta})}{\alpha_2(k,\widetilde{f_{\alpha,\beta}})}\cdot r(k,\widetilde{f_{\alpha,\beta}}).
\end{equation}
Since 
$$
r(2k,f_{0,1})=r(k,f_{0,0}) \ \text{and} \  r(2k+1,f_{0,1})=\frac13 r(4k+2,f_{0,0}),
$$
 $r(k,\gen(f_{\alpha,\beta}))$ (and $r(k,f_{\alpha,\beta})$ itself, if the class number of $f_{\alpha,\beta}$ is one)  can be written by using the number of representations of a sum of three squares.   Note that
$$
h(f_{\alpha,\beta})=1 \quad \text {if and only if} \quad \alpha,\beta \in \{1,2,4,8\} \ \ \text{or}\ \ (\alpha,\beta)=(2,16), \ (8,16).
$$

A similar argument can be applied when the spinor genus of $f_{\alpha,\beta}$ contains only one class. We show that this happens exactly when 
$$
(\alpha,\beta)=(1,16), \ (4,16), \ (8,64) \  \text{and} \ (16,16).
$$
Furthermore, we provide an exact formula  for $r(k,f_{\alpha,\beta})$ in these cases by using the number of representations of a sum of three squares.

In Section 3, we consider the representations of integers by a quadratic form with some congruence conditions. Let $n$ be a positive integer and let
$$
f(x_1,x_2,\dots,x_n)=\sum_{i,j=1}^n f_{ij} x_ix_j \quad (f_{i,j}=f_{ji} \in \z)
$$
be a positive definite quadratic form of rank $n$. Let $M_f=(f_{ij}) \in M_{n\times n}(\z)$ be the corresponding symmetric matrix.  
For a matrix    $B \in M_{n \times n}(\z)$ and $\mathbf w=(w_1,w_2,\dots,w_n), \mathbf s=(s_1,s_2,\dots,s_n)$, we define, for an integer $a$, 
$$
R_{B,\mathbf w}^{\mathbf s}(a,M_f)=\{ \mathbf x \in \z^n : \mathbf x^tM_f\mathbf x=a, \  B\mathbf x \equiv \mathbf w \pmod {\mathbf s}\},
$$
and for a quadratic form $g$ of rank $n$ whose corresponding symmetric matrix is $M_g$ and a vector $\mathbf z \in \z^n$,  we define, for an integer $b$,
$$
R(b;\mathbf z,M_g)=\{ \mathbf x \in \z^n : \mathbf x^tM_g\mathbf x+2\mathbf x^t \mathbf z=b\}.
$$
We show that there is a one to one correspondence between these two sets  if we choose parameters in each set suitably.  We also show that in some particular cases, for example, $\mathbf w=\mathbf 0$, representations of a quadratic form with some congruence conditions can be interpreted as  representations of a subform without congruence condition. If the corresponding subform has class number one or the spinor genus of the subform  contains only one class, then we may give a formula on the number of  representations of quadratic forms with some congruence conditions. By using this method, we reprove Theorems (1.9), (1,10) and (1,11) in \cite{be}, and  prove some extensions of  them. 

A (quadratic) $\z$-lattice $L=\z\mathbf x_1+\z \mathbf x_2+\dots+\z \mathbf x_n$ of rank $n$ is a free $\z$-module equipped with a bilinear form $B : L\times L \to \z$. We define the quadratic form corresponding to $L$ by  $f_L(x_1,x_2,\dots x_n)=\sum_{i,j=1}^n B(\mathbf x_i,\mathbf x_j)x_ix_j$.
 We also define the corresponding symmetric matrix $M_L=(B(\mathbf x_i,\mathbf x_j)) \in M_{n\times n}(\z)$. Note that these three terminologies are equivalent with each other. 
  If $M_L$ is diagonal, we briefly write
 $$
 L=\langle a_{11},a_{22},\dots,a_{nn}\rangle.
 $$
 
 In this article, we always assume that any quadratic form is {\it positive definite}. 

Any unexplained notations and terminologies can be found in \cite{ki} or \cite{om}.


\section{representations of generalized Bell ternary quadratic forms}
A ternary $\mathbb Z$-lattice $L$ is said to be a generalized Bell ternary $\mathbb Z$-lattice if  $L$ is isometric to   $\langle 1,2^{\alpha},2^{\beta} \rangle \ ( \alpha \leq \beta)$, for some non negative integers $\alpha, \beta$.
As noted in the introduction,  it is well known that there are exactly 12 generalized Bell ternary $\mathbb Z$-lattices having class number 1. 
For each  of these 12 lattices,  it is proved in \cite {b} and \cite {cl}   that the number of representations of an integer $k$  can be written as a constant multiple of  the number of representations of an integer, which is not necessarily same to $k$,  by a sum of three squares.
In this section, we prove similar results in the case when the spinor genus of a generalized Bell ternary lattice contains only one class.  

Let $L$ be a (positive definite integral) ternary $\mathbb Z$-lattice and let $p$ be a prime.
We define  a $\Lambda_p$-transformation as follows:
$$
\Lambda_p(L) = \{ \mathbf x \in L \mid Q(\mathbf x+\mathbf z) \equiv Q(\mathbf z) ~ (\text{mod} ~ p) ~ \text{for all} ~ \mathbf z \in L\}.
$$
Let $\lambda_p(L)$ be the primitive lattice obtained from $\Lambda_p(L)$ by scaling $V =  L\otimes \mathbb Q$ by a suitable rational number.

For $L'\in\text{gen}(L)~(L'\in\text{spn}(L))$ and any prime $p$, one may easily show that $\lambda_p(L')\in\text{gen}(\lambda_p(L))~(\lambda_p(L')\in\text{spn}(\lambda_p(L)), \text{~respectively})$. 
If we define $\text{gen}(L)/\sim$ the set of all classes in $\text{gen}(L)$, it is well known that  the map 
\begin{equation}\label{lambda}
\lambda_p : \text{gen}(L)/\sim~ \longrightarrow~ \text{gen}(\lambda_p(L))/\sim
\end{equation}
given by $[L']\mapsto[\lambda_p(L')]$ for any class $[L']\in\text{gen}(L)/\sim$ is well-defined and surjective (see \cite {wa}).
Furthermore,  the restriction map $\lambda_p$ to $\text{spn}(L)/\sim$ is also a surjective map onto $\text{spn}(\lambda_p(L))/\sim$. 
Hence 
$$
h(L)\geq h(\lambda_p(L)) \quad \text{and}  \quad h_s(L)\geq h_s(\lambda_p(L))
$$
 for any prime $p$, where $h(L)$ ($h_s(L)$) is the number of classes in the genus (spinor genus, respectively) of $L$.

\begin{lem}\label{Bell form}
For a generalized Bell ternary $\mathbb Z$-lattice $L$,  $h_s(L)=1$ and $h(L) \neq 1$ if and only if $L$ is isometric to one of the following $4$ lattices:
$$
L_1=\langle1,1,16\rangle, \quad L_2=\langle1,4,16\rangle,\quad L_3=\langle1,8,64\rangle,\quad L_4=\langle1,16,16\rangle.
$$
\end{lem}

\begin{proof}
Let $L  \simeq \langle1,2^\alpha,2^\beta\rangle \ (\alpha \le\beta)$ be a generalized Bell ternary lattice.  Assume that $\alpha +4 \leq \beta $ if $\alpha$ is odd, $\alpha +5 \leq \beta$ otherwise.  
Then $L$ can be transformed to one of  two lattices:
$$
\langle1,1,32\rangle, \quad \langle1,2,32\rangle
$$
by taking finite number of $\lambda_2$-transformations. Since
$$
h_s(\langle1,1,32\rangle) = h_s(\langle1,2,32\rangle) = 2,
$$
we have $h_s(L)\ge 2$ in this case.  Hence we may assume that $\beta \leq \alpha + 3$ if $\alpha$ is odd, $\beta \leq \alpha +4$ otherwise. 
Assume further that $\alpha \geq 5$.  Then $L$ can be transformed to one of the following lattices:
$$
\langle 1,2^5, 2^{\gamma} \rangle,\quad \langle 1,2^6,2^{\delta} \rangle,
$$
 where $5 \leq \gamma \leq 8$ and $6 \leq \delta \leq 10$,  by taking finite number of $\lambda_2$-transformations to $L$.  One can easily compute that
$$
h_s(\langle 1,2^5,2^{\gamma} \rangle)\geq 2,  \quad
h_s( \langle 1,2^6,2^{\delta} \rangle) \geq 2, 
$$
for any $5 \leq \gamma \leq 8$ and $6 \leq \delta \leq 10$.  Therefore we have $h_s(L)\ge 2$ in this case.
For the remaining 23 candidates, one can easily show that $h(L)=1$ or $h_s(L) \geq 2$ except the following 4 lattices
$$
\langle1,1,16\rangle, \quad \langle1,4,16\rangle,\quad \langle1,8,64\rangle,\quad \langle1,16,16\rangle.
$$ 
One may easily show that all of these 4 lattices have class number $2$ and their genera contain 2 spinor genera. This completes the proof.
\end{proof}

\begin{rmk}\label{other} 
{\rm  For each lattice $L_i$ defined above,  the other spinor genus in the genus of $L_i$ contains only one class  $[L_i']$ defined by 
$$
L_1'=\begin{pmatrix} 2&0&1\\0&2&1\\1&1&5\end{pmatrix}, \   L_2'=\begin{pmatrix} 4&0&0\\0&4&2\\0&2&5\end{pmatrix}, \   L_3'=\begin{pmatrix} 4&0&2\\0&8&0\\2&0&17\end{pmatrix}, \   L_4'=\begin{pmatrix} 4&2&2\\2&9&1\\2&1&9\end{pmatrix}. \
$$ }  
\end{rmk} 

We will show that the number of representations of an integer $k$ by each lattice $L_i$ in Lemma \ref{Bell form} can be written by using the number of representations of the integer $k$ by the lattice $\langle1,1,1\rangle$ or $\langle1,1,2\rangle$. 
For any non negative integer $k$, we define 
$$
\mathfrak{r}_1(k) = r(k, \langle1,1,1\rangle) \quad\text{and}\quad \mathfrak{r}_2(k) = r(k, \langle1,1,2\rangle).
$$
\begin{lem}\label{multiple}
Let $L_i$$(1\leq i\leq 4)$ be  a $\mathbb Z$-lattice in Lemma \ref{Bell form} and let $k=2^a(8t+\alpha)$ be an integer such that $a,t\in\mathbb{N}\cup\{0\}$ and $\alpha\in\{1,3,5,7\}$. Then we have 
$$
\begin{array}{ll}
 r(k,\text{gen}(L_i))=\begin{cases}c_i(a,\alpha)\cdot\mathfrak{r}_1(k) &\text{if $i=1,2,4,$}\\
c_i(a,\alpha)\cdot\mathfrak{r}_2(k) &\text{otherwise,}\end{cases}
\end{array}
$$
where the constant  $c=c_i(a,\alpha)$ depends only on $i,a$ and $\alpha$.
The values of  $c_i(a,\alpha)$ are given in Table $1$.  
\end{lem}

\begin{proof}
The lemma follows directly from the equation \eqref{minsi-3}. For the computations of local densities, see \cite {ya}.
\end{proof}

\begin{center}
\renewcommand{\arraystretch}{1.3}
\begin{tabular}{ll|ll|ll|llll}
\multicolumn{8}{l}{Table 1 Values of $c_i(a,\alpha)$}\\
\hline
$i=1$ & & $i=2$ & & $i=3$ &  &$i=4$& &\\ 
$(a,\alpha)$ & $c$ & $(a,\alpha)$ & $c$ & $(a,\alpha)$ & $c$ & $(a,\alpha)$ & \!$c$\\ \hline\hline
\!$(0,1)$ & \!$\frac13$ & $(0,1)$ & \!$\frac16$ & $(0,1)$ & \!$\frac14$ & $(0,1)$ & \!$\frac16$\\ 
\!$(0,5)$ & \!$\frac13$ & $(0,5)$ & \!$\frac16$ & $(2,1)$ & \!$\frac16$ & $(2,1)$ & \!$\frac13$\\
\!$(1,1)$ & \!$\frac13$ & $(2,1)$ & \!$\frac23$ & $(2,3)$ & \!$\frac16$ & $(2,5)$  & \!$\frac13$\\  
\!$(1,5)$ & \!$\frac13$ & $(2,5)$ & \!$\frac23$ & $(3,1)$ &  \!$\frac13$ & $(a,\alpha),\ (a\ge4)$  & \!$1$\\  
\!$(2,1)$ & \!$\frac23$ & $(3,\alpha)$ & \!$\frac13$ & $(3,3)$ & \!$\frac12$ &\text{otherwise}  & \!$0$\\
\!$(2,5)$ & \!$\frac23$ & $(a,\alpha),\ (a\ge4)$ & \!$1$ & $(4,\alpha)$ &  \!$\frac16$ & &\\
\!$(3,\alpha)$ & \!$\frac13$ &\text{otherwise} & \!$0$ & $(5,1)$ & \!$\frac13$ &  &\\
\!$(a,\alpha),\ (a\ge4)$ & \!$1$ & & &$(5,3)$ &\!$1$ &  &\\
\!\text{otherwise}&\!$0$ & & & $(5,5)$ & \!$\frac13$ & &\\  
 &&&&$(6,\alpha)$ & \!$\frac13$ &  &\\
&&&&$(a,\alpha),\ (a\ge7)$&\!$1$&&\\
&&&&\text{otherwise}&\!$0$&&&\\
\hline
\end{tabular}
\end{center} 

\newpage

\begin{thm}\label{main} 
Let $L_i ~ (1\leq i \leq 4)$ be a ternary $\z$-lattice in Lemma \ref{Bell form}.  For any integer $k=2^a(8t+\alpha)$ such that $a,t\in\mathbb{N}\cup\{0\}$ and $\alpha\in\{1,3,5,7\}$,
 we have 

$$
\begin{array}{ll}
r(k,L_1)= \begin{cases} 
\frac13 \mathfrak{r}_1(k) + \delta_{\square}(k)\cdot (-1)^{\frac{\sqrt{k}-1}{2}} \cdot 2\sqrt{k} &\text{if  $(a,\alpha) = (0,1)$}, \\
c_1(a,\alpha) \cdot \mathfrak{r}_1(k)  &\text{otherwise},
\end{cases} 
\\
\\
r(k,L_2)= \begin{cases}
\frac16 \mathfrak{r}_1(k) + \delta_{\square}(k)\cdot (-1)^{\frac{\sqrt{k}-1}{2}} \cdot \sqrt{k} &\text{if $(a,\alpha) = (0,1)$}, \\
c_2(a,\alpha) \cdot \mathfrak{r}_1(k)  &\text{otherwise},
\end{cases} 
\\
\\
r(k,L_3)= \begin{cases} 
\frac14 \mathfrak{r}_2(k) + \delta_{\square}(k)\cdot (-1)^{\frac{\sqrt{k}-1}{2}} \cdot (-1)^{\frac18 (k-1)} \cdot \sqrt{k} &\text{if  $(a,\alpha) = (0,1)$},\\
c_3(a,\alpha)\cdot \mathfrak{r}_2(k)  &\text{otherwise},
\end{cases} 
\\
\\
r(k,L_4)= \begin{cases} 
\frac16 \mathfrak{r}_1(k) + \delta_{\square}(k)\cdot (-1)^{\frac{\sqrt{k}-1}{2}} \cdot \sqrt{k} &\text{if $(a,\alpha) = (0,1)$},\\
c_4(a,\alpha) \cdot \mathfrak{r}_1(k)  &\text{otherwise},
\end{cases} 
\end{array}
$$
\\
where $\delta_{\square}(k)=1$ if $k$ is a square of an integer,  $\delta_{\square}(k)=0$ otherwise.
\end{thm}

\begin{proof}
Since proofs are quite similar to each other, we only provide the proof of the third case. The genus of $L_3$ consists of the following two lattices up to isometry:
$$
L_3=\langle 1,8,64 \rangle \quad\text{and}\quad L'_3= \begin{pmatrix} 4&0&2\\0&8&0\\2&0&17\end{pmatrix}.
$$
Note that $o(L_3) = o(L'_3)=8$. 
By Lemma \ref{multiple}, we have
\begin{equation}\label{1}
\frac{r(k,L_3)}{2} + \frac{r(k,L'_3)}{2} =c_3(a,\alpha) \cdot \mathfrak{r}_2(k),
\end{equation}
for any non negative integer $k$.  

On the other hand,  one may easily compute the difference 
$$
r(k, \text{spn}(L_3)) - r(k,\text{spn}(L_3'))
$$
by using Korollar 2 of \cite {sp2}.
Since $h_s(L_3)=h_s(L'_3)=1$ by Lemma \ref{Bell form}, we have
\begin{equation}\label{3}
r(k,L_3)-r(k,L'_3)=\begin{cases}  \delta_{\square}(k)\cdot (-1)^{\frac{\sqrt{k}-1}{2}} \cdot (-1)^{\frac18 (k-1)} \cdot 2\sqrt{k} &\text{if $k$ is odd},\\
0   &\text{otherwise}.\end{cases}
\end{equation}
From the equations \eqref{1}, and \eqref{3} we have the third equality in the theorem.
\end{proof}
\section{Representations of quadratic forms with some congruence conditions}
In this section, we consider representations of a quadratic form with some congruence conditions. We show that there is a one to one correspondence between representations of a quadratic form with some congruence conditions  and representations of a quadratic polynomial which is suitably chosen.  We also show that in some particular cases, representations of a quadratic form with some congruence conditions can be interpreted as  representations of a subform without congruence condition. If the corresponding subform has class number one or the spinor genus of the subform  contains only one class, then we may give a formula on the number of  representations of a quadratic form with some congruence conditions. By using this method, We reprove  Theorems (1.9), (1.10) and (1.11) in \cite{be},  and  prove some extensions of them. From now on, $\z^n$ denotes the set of $n\times 1$ column vectors. 

Let $L$ be a $\mathbb Z$-lattice of rank $n$ with a matrix presentation $M$ and let  $a$ be an integer.  
For an integral matrix $B=(b_{ij}) \in M_{n\times n}(\z)$ and vectors $\mathbf w=(w_1,w_2,\dots,w_n)^t \in\mathbb Z^n$, $\mathbf s=(s_1,s_2,\dots,s_n)^t\in\mathbb Z^n$, we define
$$
\mathcal M_{B,\mathbf w}^{\mathbf s} =\{ \mathbf x \in \z^n : B\mathbf x  \equiv \mathbf w \pmod {\mathbf s}\}.
$$
Here, for any two vectors $\mathbf w=(w_1,w_2,\dots,w_n)^t$ and $\mathbf w'=(w_1',w_2',\dots,w_n')^t$ in $\z^n$, we say $\mathbf w \equiv \mathbf  w' \pmod{\mathbf s}$ if $w_i \equiv w_i' \pmod{s_i}$ for any $i=1,2,\dots,n$. We always assume that  $\mathbf y \in \mathcal M_{B,\mathbf w}^{\mathbf s} \ne \emptyset$. Clearly, $\mathbf y+\mathcal M_{B,\mathbf 0}^{\mathbf s}=\mathcal M_{B,\mathbf w}^{\mathbf s}$. Since $\mbz$ is a free $\z$-module of rank $n$, there is a basis $\{\mathbf u_i \in \z^n\}_{i=1}^n$ for $\mbz$ such that $\mbz=\z \mathbf u_1+\z \mathbf u_2+\cdots+\z \mathbf u_n$. Define $\mathcal U_{B,\mathbf 0}^{\mathbf s}=[\mathbf u_1,\mathbf u_2,\dots,\mathbf u_n] \in M_{n\times n}(\z)$ so that $\text{Im}(\mathcal U_{B,\mathbf 0}^{\mathbf s})=\mbz$. 
We also define the smallest free $\z$-module containing the left coset $\mbw$ by $\widetilde{\mbw}$, which is, in fact,
$$
\widetilde{\mbw}=\z\mathbf y+\mbz.
$$
We define  $d_{B,\mathbf w}^{\mathbf s}=[ \widetilde{\mbw},\mbz]$. Note that $d_{B,\mathbf w}^{\mathbf s}$ is the smallest positive integer such that $d_{B,\mathbf w}^{\mathbf s}\cdot\mathbf y \in \mathcal M_{B,\mathbf 0}^{\mathbf s}$.
 For a fixed basis $\{\mathbf v_i\}_{i=1}^n$ for $\widetilde{\mbw}$, we define $\widetilde{\mathcal V_{B,\mathbf w}^{\mathbf s}}=[\mathbf v_1,\mathbf v_2,\dots,\mathbf v_n] \in  M_{n \times n}(\z)$.   Finally, we define quadratic forms
$$
M_{B,\mathbf 0}^{\mathbf s}= (\mathcal U_{B,\mathbf 0}^{\mathbf s})^t M\mathcal U_{B,\mathbf 0}^{\mathbf s}\quad \text{and} \quad \widetilde{M_{B,\mathbf w}^{\mathbf s}}=(\widetilde{\mathcal V_{B,\mathbf w}^{\mathbf s}})^tM\widetilde{\mathcal V_{B,\mathbf w}^{\mathbf s}}.
$$
Note that $\widetilde{M_{B,\mathbf w}^{\mathbf s}}$ is independent of the choice of basis $\{\mathbf v_i\}_{i=1}^n$ up to isometry.

 In this section, we consider the set
 $$
R_{B,\mathbf w}^{\mathbf s} (a,M)=\{\mathbf x\in \mathcal M_{B,\mathbf w}^{\mathbf s}: {\mathbf x}^tM{\mathbf x}=a\} \ \text{and} \  r_{B,\mathbf w}^{\mathbf s} (a,M)=\vert R_{B,\mathbf w}^{\mathbf s} (a,M)\vert.
$$
For an integer $b$, a vector $\mathbf z \in \z^n$ and a quadratic form $N$, we define 
$$
R(a;\mathbf z,N)=\{ \mathbf x \in \z^n :{\mathbf x}^tN{\mathbf x}+2\mathbf x^t\mathbf z=a\}.
$$
The following lemma says that there is a one to one correspondence between representations of a quadratic form with congruence conditions and representations of a quadratic polynomial which is suitably chosen. 
\begin{thm}  \label{corr}  Let $M, a, B, \mathbf w, \mathbf s$  and $\mathbf y$  be  given as above.  
Then the map $\Phi: \mathbf x \rightarrow \mathcal (U_{B,\mathbf 0}^{\mathbf s})^{-1}(\mathbf x-\mathbf y)$ from $R_{B,\mathbf w}^{\mathbf s}(a,M)$ to $R(a-\mathbf y^tM\mathbf y;(\mathcal U_{B,\mathbf 0}^{\mathbf s})^tM\mathbf y,M_{B,\mathbf 0}^{\mathbf s})$ is a bijective map.  Conversely, for a set $R(a;\mathbf z,N)$, there are  $M',a', B',\mathbf w',\mathbf s'$  and $\mathbf y'$ such that  
$$
\Phi(R_{B',\mathbf w'}^{\mathbf s'}(a',M'))=R(a;\mathbf z,N).
$$
  \end{thm} 

\begin{proof}
Let $\mathbf x\in R_{B,\mathbf w}^{\mathbf s}(a,M)$. Since $\mathbf x-\mathbf y \in \mbz$, we have $(\mathcal U_{B,\mathbf 0}^{\mathbf s})^{-1}(\mathbf x-\mathbf y)\in\mathbb Z^n$.
Hence one may easily check that 
$$
(\mathcal U_{B,\mathbf 0}^{\mathbf s})^{-1}(\mathbf x-\mathbf y)\in R(a-\mathbf y^tM\mathbf y;(\mathcal U_{B,\mathbf 0}^{\mathbf s})^tM\mathbf y,M_{B,\mathbf 0}^{\mathbf s}).
$$ 
Therefore the map $\Phi$ is well-defined.  For any $\mathbf z \in R(a-\mathbf y^tM\mathbf y;(\mathcal U_{B,\mathbf 0}^{\mathbf s})^tM\mathbf y,M_{B,\mathbf 0}^{\mathbf s})$, if we define $\Psi(\mathbf z)=U_{B,\mathbf 0}^{\mathbf s}\mathbf z+\mathbf y$, then $\Phi \circ \Psi =\Psi\circ\Phi=\text{Id}$.

To prove the converse, let $N=[\mathbf t_1,\mathbf t_2,\dots,\mathbf t_n]$. 
By Invariant Factor Theorem, there is a basis $\{\mathbf e_i\}_{i=1}^n$ for $\z^n$  and integers $s_i$ for $1\le i\le n$ such that $s_i \mid s_{i+1}$ and 
$$
\z \mathbf t_1+\z \mathbf t_2+\cdots+\z \mathbf t_n=\z(s_1 \mathbf e_1)+\z(s_2\mathbf e_2)+\cdots+\z(s_n\mathbf  e_n).
$$
If we define $B'=[\mathbf e_1,\mathbf e_2,\dots,\mathbf e_n]^{-1} \in M_{n\times n}(\z)$ and $\mathbf s'=(s_1',s_2',\dots,s_n')$, then one may easily show that 
$$
\z \mathbf t_1+\z \mathbf t_2+\cdots+\z\mathbf t_n=\mathcal M_{B',\mathbf 0}^{\mathbf s'}\quad \text{and} \quad N=\mathcal U_{B',\mathbf 0}^{\mathbf s'}.
$$
If we define $\mathbf y'=\mathbf z$, $\mathbf w'=B'\mathbf z$,
$$
  M'=\det(N)N^{-1}\quad \text{and} \quad a'=(\det N)a+\mathbf z^t\det(N)N^{-1}\mathbf z,
$$
then one may easily check the above map $\Phi$  is a bijective map from $R_{B',\mathbf w'}^{\mathbf s'}(a',M')$ to $R(a;\mathbf z,N)$.
\end{proof}

\begin{lem}  \label{formul} Under the same notations given above, we have 
$$
r(a,\widetilde{M_{B,\mathbf w}^{\mathbf s}})=r_{B,\mathbf 0}^{\mathbf s}(a,M)+\sum_{k=1}^{d_{B,\mathbf w}^{\mathbf s}-1} r_{B,k\mathbf w}^{\mathbf s}(a,M).
$$ 
\end{lem}

\begin{proof} The lemma follows directly from the fact that the $\z$-module $\widetilde {\mathcal M_{B,\mathbf w}^{\mathbf s}}$ is a disjoint union of left cosets $k\mathbf y+\mathcal M_{B,\mathbf 0}^{\mathbf s}=\mathcal M_{B,k\mathbf w}^{\mathbf s}$ for $k=0,1,\dots,d_{B,\mathbf w}^{\mathbf s}\!\!-1$. 
\end{proof}
\begin{cor}  \label{want}  Under the same notations given above, we have the followings:
\begin{itemize}
\item [(i)]  We have $r_{B,\mathbf 0}^{\mathbf s}(a,M)=r(a,M_{B,\mathbf 0}^{\mathbf s})$.
\item [(ii)] If $d_{B,\mathbf w}^{\mathbf s}=2$, then $r_{B,\mathbf w}^{\mathbf s}(a,M)=r(a,\widetilde{M_{B,\mathbf w}^{\mathbf s}})-r(a,M_{B,\mathbf 0}^{\mathbf s})$.
\item [(iii)]   The map $\mathbf x \to -\mathbf x$ from $R_{B,\mathbf w}^{\mathbf s}(a,M)$ 
to $R_{B,-\mathbf w}^{\mathbf s}(a,M)$ is bijective.  In particular, if $d_{B,\mathbf w}^{\mathbf s}=3$, then $r_{B,\mathbf w}^{\mathbf s}(a,M)=\displaystyle \frac 12\left(r(a,\widetilde{M_{B,\mathbf w}^{\mathbf s}})-r(a,M_{B,\mathbf 0}^{\mathbf s})\right)$.
\item [(iv)] Assume that for $s=\gcd(s_1,s_2,\dots,s_n)$, $\gcd(\det B,s)=\gcd(a,s)=1$. Then for any $k$ such that $\gcd(s,k)>1$, $r_{B,k\mathbf w}^{\mathbf s}(a,M)=0$.
\end{itemize}
\end{cor} 

\begin{proof} If $\mathbf w=\mathbf 0$, then we can take $\mathbf y=\mathbf 0$. Hence the first assertion comes directly from Theorem \ref{corr}. The second assertion comes directly from Lemma \ref{formul}. The  third and fourth assertions are trivial. 
\end{proof} 

Now, by using the above corollary and spinor representation theory of quadratic forms,  we reprove Theorems (1.9), (1.10) and (1,11) in \cite{be}. 
To do these, we define ternary quadratic forms:
$$
M_1=\langle1,1,1\rangle, \quad M_2=\langle1,1,2\rangle,\quad M_3=\langle1,4,12\rangle.
$$
Let  
$$
B_1=\begin{pmatrix} 1&0&0\\ 0&1&0\\ 0&0&1\end{pmatrix}, \quad B_2=\begin{pmatrix} 1&0&0\\ 0&1&0\\ 0&0&1\end{pmatrix}, \quad B_3=\begin{pmatrix} 1&0&0\\ 0&1&-1\\ 0&1&1\end{pmatrix}.
$$
and
$$
\begin{array}{rrr}
\mathbf w_1=(1,2,2)^t,  &\mathbf w_2=(1,4,0)^t, &\mathbf w_3=(3,0,2)^t,\\
\mathbf s_1=(4,8,8)^t,   &\mathbf s_2=(4,16,2)^t, &\mathbf s_3=(12,6,6)^t.
\end{array}
$$
\begin{cor} \label{berkovich}
Let $n$ be a positive integer. Then we have
\begin{itemize}

\item [(i)]$r_{B_1,\mathbf w_1}^{\mathbf s_1}(8n+1,M_1) = 0$  if and only if $8n+1=M^2$ and all prime divisors of $M$ are congruent to $1$ modulo $4$. 

\item [(ii)]$r_{B_2,\mathbf w_2}^{\mathbf s_2}(8n+1,M_2) = 0$  if and only if $8n+1=E^2$ and all prime divisors of $E$ are congruent to $1$ or $3$ modulo $8$.

\item [(iii)]$r_{B_3,\mathbf w_3}^{\mathbf s_3}(24n+1,M_3) = 0$ if and only if $24n+1=W^2$ and all prime divisors of $W$ are congruent to $1$ modulo $3$.

\end{itemize}
\end{cor}

\begin{proof}
First, we define quadratic forms
$$
K_1=\begin{pmatrix} 9&4&2\\ 4&16&8\\ 2&8&36\end{pmatrix},\quad K_2=\begin{pmatrix} 4&0&2\\ 0&8&0\\ 2&0&17\end{pmatrix},\quad K_3=\begin{pmatrix} 9&0&0\\ 0&16&8\\ 0&8&112\end{pmatrix}.
$$
One may easily show that the spinor genus of $K_i$ contains only one class for each $i=1,2,3$. 
 Note that 
$\mathcal M_{B_1,\mathbf 0}^{\mathbf s_1}=\z(4,0,0)^t+\z(0,8,0)^t+\z(0,0,8)^t $ and
$$
 \widetilde{\mathcal M_{B_1,\mathbf w_1}^{\mathbf s_1}}=\z(1,2,2)^t+\z(4,0,0)^t+\z(0,8,0)^t. 
$$
Hence $d_{B_1,\mathbf w_1}^{\mathbf s_1}=4$ in this case, and one may easily show that 
$$
M_{B_1,\mathbf 0}^{\mathbf s_1}=\begin{pmatrix} 16&0&0\\0&64&0\\0&0&64\end{pmatrix} \ \ \text{and} \ \   \widetilde{M_{B_1,\mathbf w_1}^{\mathbf s_1}}\simeq K_1. 
$$ 
Since $r(8n+1,M_{B_1,\mathbf 0}^{\mathbf s_1})=0$,  we have $2r_{B_1,\mathbf w_1}^{\mathbf s_1}(8n+1,M_1)=r(8n+1,K_1)$ by (iii) and (iv) of Corollary \ref{want}. 
 Since the  spinor genus of $K_1$ contains only one class, it represents all integers of the form $8n+1$ except spinor exceptional integers of it.  
Hence one may easily show that $r(8n+1,K_1) =0$ if and only if $8n+1=M^2$ and all prime divisors of $M$ are congruent to $1$ modulo $4$.

We consider the second case which corresponds to Theorem (1.10) of \cite{be}. Note that 
$$
\begin{array} {ll}
& r_{B_2,\mathbf w_2}^{\mathbf s_2}(8n+1,M_2)\\
&=\vert \{ (x,y,z)^t \in \z^3  :  x^2\!+\!y^2\!+\!2z^2=8n+1, (x,y,z)\equiv (1,4,0)  (\text{mod } (4,16,2) )\}\vert \\
&=\vert \{ (x,y,z)^t \in \z^3 : (4x+1)^2+(16y+4)^2+2(2z)^2=8n+1\}\vert \\ 
&=\vert \{ (x,y,z)^t \in \z^3 : (4x+1)^2+16(4y+1)^2+8z^2=8n+1\}\vert \\
&=\frac 14 \vert \{ (x,y,z)^t \in \z^3 : (2x+1)^2+16(2y+1)^2+8z^2=8n+1\}\vert \\
&=\frac14 \vert \{ (x,y,z)^t \in \z^3  : x^2\!+\!16y^2\!+\!8z^2\!=\!8n+1, \!(x,y,z)\!\equiv \!(1,1,0)  (\text{mod } (2,2,1) )\}\vert. \\
\end{array}
$$
Hence if we define $\mathbf w_2'=(1,1,0)$, $\mathbf s_2'=(2,2,1)$ and $M_2'=\langle 1,16,8\rangle$, then 
$$
r_{B_2,\mathbf w_2}^{\mathbf s_2}(8n+1,M_2)=\frac 14 r_{B_2,\mathbf w_2'}^{\mathbf s_2'}(8n+1,M_2').
$$
Note that 
$$
M_{B_2,\mathbf 0}^{\mathbf s_2'}=\langle 4,64,8\rangle \qquad \text{and} \qquad \widetilde{M_{B_2,\mathbf \mathbf w_2'}^{\mathbf s_2'}}\simeq K_2.
$$
Therefore by Corollary \ref{want},  we have
$$
r_{B_2,\mathbf w_2}^{\mathbf s_2}(8n+1,M_2)=\frac14 r(8n+1,K_2).
$$
The assertion follows directly from the fact that the spinor genus of $K_2$ contains only one class.  

Now, we consider the third case. One may easily show that 
$\mathcal M_{B_3,\mathbf 0}^{\mathbf s_3}=\z(12,0,0)^t\!+\z(0,3,-3)^t+\z(0,3,3)^t $ and
$$
 \widetilde{\mathcal M_{B_3,\mathbf w_3}^{\mathbf s_3}}=\z(3,0,0)^t+\z(0,3,-3)^t+\z(0,1,1)^t. 
$$
Hence $d_{B_3,\mathbf w_3}^{\mathbf s_3}=12$ in this case, and one may also  show that 
$$
M_{B_3,\mathbf 0}^{\mathbf s_3}=\begin{pmatrix} 144&0&0\\0&144&-72\\0&-72&144\end{pmatrix} \ \ \text{and} \ \   \widetilde{M_{B_3,\mathbf w_3}^{\mathbf s_3}}\simeq K_3. 
$$ 
Note that $r_{B_3,k\mathbf w_3}^{\mathbf s_3}(24n+1,M_3)=0$  for any $k$ such that $(k,12)\ne1$ by (iv) of Corollary \ref{want}. Since the map $(x_1,x_2,x_3) \to (x_1,-x_2,-x_3)$ from $R_{B_3,\mathbf w_3}^{\mathbf s_3}(24n+1,M_3)$ to $R_{B_3,5\mathbf w_3}^{\mathbf s_3}(24n+1,M_3)$  is bijective,  Lemma \ref{formul} implies that 
$$
4r_{B_3,\mathbf w_3}^{\mathbf s_3}(24n+1,M_3)=r(24n+1,K_3). 
$$
Since the  spinor genus of $K_3$ contains only one class, it represents all integers of the form $24n+1$ except spinor exceptional integers of it.  
Hence one may easily show that $r(24n+1,K_3) =0$ if and only if $24n+1=W^2$ and all prime divisors of $W$ are congruent to $1$ modulo $3$.  \end{proof}

\begin{rmk} {\rm Since the spinor genus of each lattice $K_i$ contains only one class, we may  give a formula on   $r_{B_i,\mathbf w_i}^{\mathbf s_i}(8n+1,M_i)$ by using the number of representations of $K_i$  (and also $M_i$) for any $i=1,2,3$. For example, we have
$$
\begin{array} {ll} 
\!\!r_{B_2,\mathbf w_2}^{\mathbf s_2}(8n+1,M_2) \!\!\! &\!\!=\frac14 r(8n+1,K_2)=\frac 14 r(8n+1,L_3')\\
                                           \!\!\! &\!\!=\frac1{16}\mathfrak{r}_2(8n+1) -\frac14  \delta_{\square}(8n+1)\cdot (-1)^{\frac{\sqrt{8n+1}-1}{2}} \cdot (-1)^{n} \cdot \sqrt{8n+1}.
\end{array}               
 $$}
\end{rmk}

If we use spinor genera having only one class,  we may have some more results similar to the above. For some examples of  ternary quadratic forms whose spinor genera consists of only one class, 
see \cite {behh}  or Section 7 of \cite {ja}. 
 
Define quadratic forms
$$
M_4=\langle1,1,1\rangle,\quad M_5=\langle1,1,1\rangle,\quad M_6=\langle1,1,2\rangle.
$$
We also define 
$$
B_4=B_5=B_6=\begin{pmatrix} 1&0&0\\ 0&1&0\\ 0&0&1\end{pmatrix}, $$
and
$$
\begin{array}{rrr}
\mathbf w_4=(1,0,2)^t,  &\mathbf w_5=(1,2,2)^t,   &\mathbf w_6=(1,1,2)^t,\\
\mathbf s_4=(2,2,4)^t,    &\mathbf s_5=(2,4,4)^t,   &\mathbf s_6=(2,2,4)^t.
\end{array}
$$

\begin{cor} \label{last1}
Let $n$ be a positive integer. Then we have
\begin{itemize}

\item [(i)]$r_{B_4,\mathbf w_4}^{\mathbf  s_4}(8n+1,M_4) =0$ if and only if $8n+1=M^2$ and all prime divisors of $M$ are congruent to $1$ modulo $4$ and 
$r_{B_4,\mathbf  w_4}^{\mathbf s_4}(8n+5,M_4) > 0.$

\item [(ii)]$r_{B_5,\mathbf w_5}^{\mathbf  s_5}(8n+1,M_5)=0$  if and only if $8n+1=M^2$ and all prime divisors of $M$ are congruent to $1$ modulo $4$.

\item [(iii)]$r_{B_6,\mathbf w_6}^{\mathbf  s_6}(16n+2,M_6)=0$  if and only if $16n+2=2M^2$ and all prime divisors of $M$ are congruent to $1$ modulo $4$ and
$r_{B_6,\mathbf w_6}^{\mathbf s_6}(16n+10,M_6) > 0$.
\end{itemize}
\end{cor}

\begin{proof}
Recall that we defined in Section 2, that 
$$
L_2'=\begin{pmatrix} 4&0&2\\ 0&4&0\\ 2&0&5\end{pmatrix},\quad L_4'=\begin{pmatrix} 4&2&2\\ 2&9&1\\ 2&1&9\end{pmatrix},\quad L_1'=\begin{pmatrix} 2&0&1\\ 0&2&1\\ 1&1&5\end{pmatrix},
$$
where each spinor genus contains only one class.  Note that
$\mathcal M_{B_4,\mathbf 0}^{\mathbf s_4}=\z(2,0,0)^t+\z(0,2,0)^t+\z(0,0,4)^t $ and
$$
 \widetilde{\mathcal M_{B_4,\mathbf w_4}^{\mathbf s_4}}=\z(1,0,2)^t+\z(2,0,0)^t+\z(0,2,0)^t. 
$$
Hence $d_{B_4,\mathbf w_4}^{\mathbf s_4}=2$ in this case, and one may easily show that 
$$
M_{B_4,\mathbf 0}^{\mathbf s_4}=\begin{pmatrix} 4&0&0\\0&4&0\\0&0&16\end{pmatrix} \ \ \text{and} \ \   \widetilde{M_{B_4,\mathbf w_4}^{\mathbf s_4}}\simeq L_2'. 
$$ 
Now everything follows from Corollary \ref{want} and the fact that the spinor genus of $K_4$ contains only one class.
The proofs of the second and the third cases are quite similar to this. One may use the fact that 
$$
d_{B_5\mathbf w_5}^{\mathbf s_5}=2, \  d_{B_6,\mathbf w_6}^{\mathbf s_6}=2  \quad \text{and} \quad    \widetilde{M_{B_5,\mathbf w_5}^{\mathbf s_5}}\simeq L_4', \ 
  \widetilde{M_{B_6,\mathbf w_6}^{\mathbf s_6}}\simeq 2L_1'. 
   $$
This completes the proof. \end{proof}

\begin{cor}
Let $n$ be a positive integer. Then we have

$$
\begin{array} {rl}
\rm {(i)} \ \ r_{B_4,\mathbf w_4}^{\mathbf  s_4}(8n+1,M_4)\!\!\!&= r(8n+1,L_2')\\
  &= \frac16\mathfrak{r}_1(8n+1) -\delta_{\square}(8n+1)\cdot (-1)^{\frac{\sqrt{8n+1}-1}{2}}\cdot \sqrt{8n+1},\\
  &\\
r_{B_4,\mathbf  w_4}^{\mathbf s_4}(8n+5,M_4) \!\!\!&= r(8n+5,L_2') = \frac16\mathfrak{r}_1(8n+5).
\end{array}
$$  

$$
\begin{array} {rl}
\rm {(ii)}  \  \   r_{B_5,\mathbf w_5}^{\mathbf  s_5}(8n+1,M_5) \!\!\!&= r(8n+1,L_4')\\
  &= \frac16\mathfrak{r}_1(8n+1) -\delta_{\square}(8n+1)\cdot (-1)^{\frac{\sqrt{8n+1}-1}{2}}\cdot \sqrt{8n+1}.
\end{array}
$$  

$$
\begin{array} {rl}
\rm {(iii)} \ \  r_{B_6,\mathbf w_6}^{\mathbf  s_6}(16n+2,M_6) \!\!\!&= r(8n+1,L_1')\\
  &= \frac13 \mathfrak{r}_1(8n+1) -\delta_{\square}(8n+1)\cdot (-1)^{\frac{\sqrt{8n+1}-1}{2}} \cdot 2\sqrt{8n+1},\\
  & \\
r_{B_6,\mathbf w_6}^{\mathbf s_6}(16n+10,M_6)\!\!\!&= r(8n+5,L_1') = \frac13\mathfrak{r}_1(8n+5).
\end{array} 
$$

\end{cor}
\begin{proof}
The corollary is a direct consequence of Remark \ref {other}, Theorem \ref {main} and Corollary \ref {last1}. 
\end{proof}

\end{document}